\documentclass{article}

\usepackage[english]{babel}

\usepackage{amssymb}
\usepackage{stmaryrd}
\usepackage{amsmath,amsfonts}
\usepackage[tikz]{bclogo}
\usepackage[top=4.34cm, left=3cm, right=3cm]{geometry}

\usepackage{lmodern}
\usepackage{textcomp}
\usepackage{mathtools, bm}
\usepackage{amssymb, bm}
\usepackage[unq]{unique}
\usepackage{comment}
\usepackage[breaklinks]{hyperref}
\usepackage[numbers]{natbib}

\usepackage{amsthm}

\usepackage{cleveref}
\usepackage[utf8]{inputenc}
\usepackage[T1]{fontenc}
\usepackage{tipa}
\usepackage{float}
\usepackage{dirtytalk}
\usepackage{pbox}
\usepackage{xcolor}
\usepackage{setspace}
\usepackage[normalem]{ulem}
\usepackage{geometry}
\usepackage{amsmath}
\usepackage{amsthm}
\usepackage{tikz}
\usepackage{amssymb}
\usepackage{multicol}
\usepackage{microtype}
\usepackage{datetime}
\usepackage{endnotes}
\usepackage{babel}
\usepackage[mathscr]{euscript}
\usepackage{bigfoot}
\usetikzlibrary {arrows.meta}
\usepackage[shortlabels]{enumitem}

\newtheorem{theorem}{Theorem}[section]
\newtheorem{lemma}[theorem]{Lemma}
\newtheorem{Obs}[theorem]{Observation}
\newtheorem{conjecture}[theorem]{Conjecture}

\newtheorem{question}[theorem]{Question}

\theoremstyle{definition}

\newcommand{\fly}[0]{\mathcal{F}}
\newcommand{\pot}[1]{\mathcal{P}(#1)}
\newcommand{\muf}[0]{\mu_{\fly}}
\newcommand{\muft}[0]{\mu_{\widetilde{\fly}}}
\newcommand{\saw}[0]{temperate }
\newcommand{\sawt}[1]{$#1$-temperate}

\title{Towards odd-sunflowers: \saw families and lightnings}
\author{Jan Petr\thanks{jp895@cam.ac.uk, Department of Pure Mathematics and Mathematical Statistics (DPMMS), University of Cambridge, Wilberforce Road, Cambridge, CB3 0WA, United Kingdom} \and Pavel Turek\thanks{pkah149@live.rhul.ac.uk, Department of Mathematics, Royal Holloway, University of London, Egham, Surrey TW20 0EX, United Kingdom}}
\date{}

\begin{document}

\maketitle

\begin{abstract}
    Motivated by odd-sunflowers, introduced recently by Frankl, Pach, and P{\'a}lv{\"o}lgyi, we initiate the study of \saw families: a family $\fly \subseteq \pot{[n]}$ is said to be \emph{\saw}if each $A \in \fly$ contains at most $|A|$ elements of $\fly$ as a proper subset.

    We show that the maximum size of a \saw family is attained by the middle two layers of the hypercube $\{0,1\}^n$. As a more general result, we obtain that the middle $t+1$ layers of the hypercube maximise the size of a family $\fly$ such that each $A \in \fly$ contains at most $\sum_{j=1}^t \binom{|A|}{j}$ elements of $\fly$ as a proper subset. Moreover, we classify all such families consisting of the maximum number of sets.
    
    In the case of intersecting \saw families, we find the maximum size and classify all intersecting \saw families consisting of the maximum number of sets for odd $n$. We also conjecture the maximum size for even $n$.
\end{abstract}

\section{Introduction}

Let $[n]=\{1,\ldots,n\}$. For a family $\fly \subseteq \pot{[n]}$ and a set $A \in \pot{[n]}$, let $\muf(A)=|\pot{A} \cap \fly|$.
Given a non-negative integer $t$, we say $\fly$ is \textit{\sawt{t}} if $\muf(A) \leq \sum_{j=0}^t \binom{|A|}{j}$ for all $A \in \fly$. Note that $\muf(A) \leq \sum_{j=0}^t \binom{|A|}{j}$ is trivially satisfied for all $A$ with $|A|\leq t$.
Of particular interest to us will be \sawt{1} families, i.e. families $\fly$ where $\muf(A) \leq 1+|A|$ for all $A \in \fly$.
For brevity, we will refer to \sawt{1} families as \textit{\saw}families. 

Simple examples of \sawt{t} families include antichains (for $t\geq0$), chains (for $t \geq 1)$ and the union of any $t+1$ consecutive layers of the hypercube (for $t \geq 0$) since any set $A$ of size $|A| \geq l$ has precisely $\sum_{j=0}^l \binom{|A|}{j}$ subsets which lie at most $l$ layers below $A$.

In this short paper, we begin the investigation of properties of \sawt{t} families by considering the following two extremal questions. Recall that a family is called \emph{intersecting} if the intersection of any two elements of it is non-empty.

\begin{question} \label{qu:nintersecting}
What is the maximum size of a \sawt{t} family on $[n]$?
\end{question}

\begin{question} \label{qu:intersecting}
What is the maximum size of an intersecting \sawt{t} family on $[n]$?
\end{question}

\subsection{Main results}

As we shall see, the answer to \Cref{qu:nintersecting} is given by the total size of the $t+1$ largest layers of the hypercube.

\begin{theorem}\label{th:nintersecting}
Let $t\leq n$ be non-negative integers. The maximum size of a \sawt{t} family on $[n]$ is $\sum_{j=0}^{t} \binom{n}{\lfloor(n-t)/2\rfloor + j}$. 
\end{theorem}

Moreover, we classify all the \sawt{t} families which attain the maximal size. We use $\binom{[n]}{k}$ to denote the set of subsets of $[n]$ of size $k$.

\begin{theorem}\label{th:nintersecting-class}
Let $t \leq n$ be non-negative integers.

\begin{enumerate}[\normalfont(a)]
\item\label{item:first}
If $t \equiv n \pmod{2}$, the only \sawt{t} family of maximal size is $\bigcup_{j=0}^{t} \binom{[n]}{(n-t)/2 + j}$.

\item\label{item:second}
If $t=n-1$, any family of size $2^n-1$ is a \sawt{t} family of maximal size.

\item\label{item:third}
If $t \not\equiv n \pmod 2$ and $t < n-1$, then there are two \sawt{t} families of maximal size: $\bigcup_{j=0}^{t} \binom{[n]}{(n-t-1)/2 + j}$ and $\bigcup_{j=0}^{t} \binom{[n]}{(n-t+1)/2 + j}$.
\end{enumerate}
\end{theorem}

The case of \Cref{qu:intersecting} turns out to be more convoluted. In this paper, we restrict our attention to the case of intersecting \saw families. In such a setting, we obtain the following answer for odd $n$.

\begin{theorem} \label{th:intersecting-odd}
The maximum size of an intersecting \saw family on $[n]$ for odd $n=2k-1$ is $\binom{2k-1}{k} + \binom{2k-1}{k+1}$.
\end{theorem}

As with the non-intersecting case, we provide the full classification of the extremal families.

\begin{theorem}\label{th:intersecting-class}
For odd $n=2k-1$ the unique intersecting \saw family of maximal size is $\binom{[2k-1]}{k}\cup\binom{[2k-1]}{k+1}$ unless $n=3$, in which case there are three additional such families. These are $\{ A \in \pot{[3]} : i\in A\}$ for each $i\in [3]$. 
\end{theorem}

For even $n$, we make the following conjecture.

\begin{conjecture} \label{conj:intersecting-even}
The maximum size of an intersecting \saw family on $[n]$ for even $n=2k$ is $\frac{1}{2}\binom{2k}{k}+\binom{2k}{k+1}+\binom{2k-1}{k+2}$.
\end{conjecture}

We believe one of the extremal constructions supporting \Cref{conj:intersecting-even} is the following one, whose name comes from the shape the contained elements resemble when drawn in the Hesse diagram of the hypercube, with rows ordered lexicographically. By a \textit{lightning}, we mean the family on $[2k]$ which contains:
\begin{itemize}
\item all elements of $\binom{[2k]}{k}$ containing $1$, and
\item all elements of $\binom{[2k]}{k+1}$, and
\item all elements of $\binom{[2k]}{k+2}$ not containing $1$.
\end{itemize}

\subsection{Background and motivation}

The motivation for investigating \saw families and the two questions above in particular comes from odd-sunflowers as introduced by Frankl, Pach and P{\'a}lv{\"o}lgyi \cite{frankl2023odd}. These are in turn inspired by sunflowers, as originally studied by Erd\H{o}s and Rado \cite{ErdosRadoSunflowers}.

Let us give a brief historical overview on sunflowers and odd-sunflowers. A family of $r \geq 3$ sets $S_1,\ldots,S_r$ is an \textit{$r$-sunflower} if $S_i \cap S_j = \bigcap_{k=1}^{r} S_k$ for all $i \neq j$. A \emph{sunflower} is a family which is an $r$-sunflower for some $r \geq 3$. A family is said to be \textit{$r$-sunflower-free} if it contains no $r$-sunflower.
Analogously, it is said to be \textit{sunflower-free} if it contains no sunflower.

In 1960, Erd\H{o}s and Rado \cite{ErdosRadoSunflowers} showed that an $r$-sunflower-free family $\fly$ where every set $A \in \fly$ satisfies $|A| \leq w$ satisfies $|\fly| < w!\cdot (r-1)^w$. Furthermore, they conjectured that there exists $c=c(r)$ such that this bound may be improved to $|\fly| < c^w$. This conjecture, known as the Sunflower Conjecture, is unresolved to this date. In view of recent developments due to Alweiss, Lovett, Wu and Zhang \cite{alweiss2021improved}, Rao \cite{Rao2020Coding}, Tao \cite{tao2020sunflower} and Bell, Chueluecha and Warnke \cite{Bell2021}, the current best known bound is of the form $O(r\cdot \log(w))^w$.

A second main direction in which sunflower-free families have been studied, and the direction closer to our questions, is that of the maximum size of a sunflower-free family $\fly \subseteq \pot{[n]}$. Denote this quantity $f(n)$, and let $\phi=\lim_{n \to \infty} f(n)^{1/n}$. Erd\H{o}s and Szemerédi \cite{ERDOS1978308} conjectured in 1978 that $\phi<2$. Naslund and Sawin \cite{naslund2017upper} finally proved this almost 40 years later, building on the techniques of Croot, Lev and Pach \cite{crootlevpach}, Ellenberg and Gijswijt \cite{ellenberggijswijt} and Tao \cite{tao2016CLP}. Naslund and Sawin gave $\phi <1.890$, which remains the state of art. In the opposite direction, the best known lower bound $\phi >1.551$ is due to Deuber, Erd\H{o}s, Gunderson, Kostochka and Meyer \cite{deubner1997}.

As variants on sunflowers, Frankl, Pach and Pálvölgyi \cite{frankl2023odd} introduced the following notion of odd-sunflowers and even-sunflowers.

A non-empty family of non-empty sets is an \emph{even-degree sunflower} or, simply, an \emph{even-sunflower}, if every element of the underlying set is contained in an even number of sets (or in none). Analogously, a family of at least two non-empty sets forms an \emph{odd-degree sunflower} or, simply, an \emph{odd-sunflower}, if every element of the underlying set is contained in an odd number of sets, or in none \cite{frankl2023odd}.

It should be noted that neither an even-sunflower nor an odd-sunflower need be a sunflower. However, any $3$-sunflower is an odd-sunflower. Thus every family that does not contain an odd-sunflower is guaranteed to be sunflower-free.

Let $f_{\mathrm{even}}(n)$ be the maximum size of a family of sets in $[n]$ that contains no even-sunflower. Analogously, let $f_{\mathrm{odd}}(n)$ be the maximum size of a family of sets in $[n]$ that contains no odd-sunflower. These two quantities exhibit a diametrically opposite behaviour. On one hand, the authors of \cite{frankl2023odd} observe that, as a variant of the ``odd-town'' problem (see, for example, a textbook by Babai and Frankl \cite{babaifrankl} for a reference), the following holds.

\begin{theorem}\label{thm:even}
$f_{\mathrm{even}}(n)=n$.
\end{theorem}

On the other hand, Frankl, Pach and P{\'a}lv{\"o}lgyi \cite{frankl2023odd} give the following exponential lower bound on $f_{\mathrm{odd}}$.

\begin{theorem}
$f_{\mathrm{odd}}(n)>1.502^n$.
\end{theorem}

However, there is no better known upper bound than the one on sunflower-free families by Naslund and Sawin \cite{naslund2017upper}, i.e., $\lim_{n \to \infty} f_{\mathrm{odd}}(n)^{1/n} \leq \phi < 1.890$.

Our two questions came from pursuit of a better upper bound for $f_{\mathrm{odd}}(n)$ and two observations.
First, note that every family $\fly$ that does not contain an odd-sunflower is a \saw family. Indeed, any set $A$ together with an even sunflower $\mathcal{E}\subseteq \pot{A}\setminus \{ A\}$ form an odd-sunflower, and thus by \Cref{thm:even}, we need $\muf(A)\leq |A|+1$ for all $A\in\fly$. Furthermore, as two disjoint sets are considered an odd-sunflower, every family that does not contain an odd-sunflower is intersecting.

\section{The maximum size of a \sawt{t} family}

The first step of our proof of \Cref{th:nintersecting} is \Cref{ob:weight}, for which the following notation will be helpful. For a family $\fly$ on $[n]$, we let $\fly_i=\{A \in \fly: |A|=i\}$. Additionally, we define the \textit{weight} of $A\in \fly$ by $w(A)=\binom{n}{|A|}$, and the weight of a family $\fly$ by $w(\fly)=\sum_{A \in \fly} w(A)$.

\begin{Obs}\label{ob:weight}
Let $\fly$ be (any) family on $[n]$ and let $\mathcal{C}$ be a random maximal chain on $[n]$ chosen uniformly at random. Then $\mathbb{E}_{\mathcal{C}}[w(\fly \cap \mathcal{C})]=|\fly|$.
\end{Obs}

\begin{proof}
We simply compute $\mathbb{E}_{\mathcal{C}}[w(\fly \cap \mathcal{C})]=\sum_{i=0}^n w([i]) \cdot \mathbb{P}_{\mathcal{C}}(\fly_i\cap\mathcal{C}\neq \emptyset)=
\sum_{i=0}^n w([i]) \cdot |\fly_i| / \binom{n}{i}=
\sum_{i=0}^n |\fly_i|=
|\fly|$.    
\end{proof}

We are now ready to prove \Cref{th:nintersecting}.

\begin{proof}[Proof of \Cref{th:nintersecting}]
First, note that the \sawt{t} family $\bigcup_{j=0}^{t} \binom{[n]}{\lfloor(n-t)/2\rfloor + j}$ has size equal to $\sum_{j=0}^{t} \binom{n}{\lfloor(n-t)/2\rfloor + j}$. It remains to show that no \sawt{t} family has larger size.

Let $\fly$ be a \sawt{t} family on $[n]$ and let $\mathcal{C}$ be a maximal chain on $[n]$ chosen uniformly at random. Define $M(\mathcal{C})$ to be the maximal set in $\mathcal{C} \cap \fly$ when the intersection is non-empty, and $M(\mathcal{C})=\star$ otherwise.

Applying \Cref{ob:weight} and using conditional expectation (with the convention that $\mathbb{E}[X\, \vert\, Y]=\mathbb{P}(U\, \vert\, Y)=0$ when $\mathbb{P}(Y)=0$), we obtain the following.

\begin{align}
|\fly|
&=\mathbb{E}_{\mathcal{C}}\Bigl[w(|\mathcal{C}\cap \fly|)\Bigr]\nonumber\\
&= \mathbb{E}_{\mathcal{C}}\Bigl[w(|\mathcal{C}\cap \fly|)\, \Big|\, M(\mathcal{C}) = {\star}\Bigr]\cdot\mathbb{P}_{\mathcal{C}}\Bigl(M(\mathcal{C}) = {\star}\Bigr)+ \sum_{A\in \fly}\mathbb{E}_{\mathcal{C}}\Bigl[w(|\mathcal{C}\cap \fly|)\, \Big|\, M(\mathcal{C}) = A\Bigr]\cdot\mathbb{P}_{\mathcal{C}}\Bigl(M(\mathcal{C}) = A\Bigr)\nonumber \\
&=\sum_{A\in \fly}\mathbb{E}_{\mathcal{C}}\Bigl[w(|\mathcal{C}\cap \fly|)\, \Big|\, M(\mathcal{C}) = A\Bigr]\cdot\mathbb{P}_{\mathcal{C}}\Bigl(M(\mathcal{C}) = A\Bigr)\nonumber\\
&\leq \max_{A\in\fly}\mathbb{E}_{\mathcal{C}}\Bigl[w(|\mathcal{C}\cap \fly|)\, \Big|\, M(\mathcal{C}) = A\Bigr] \label{eq:max}\\
&=
\max_{A\in\fly} \sum_{B \in \fly \cap \pot{A}}w(B)\cdot\mathbb{P}_{\mathcal{C}}\Bigl(B\in \mathcal{C}\, \Big|\, M(\mathcal{C}) = A\Bigr)\nonumber\\
&\leq
\max_{A\in\fly} \sum_{B \in \fly \cap \pot{A}} \left.\binom{n}{|B|}\right/\binom{|A|}{|B|}\nonumber
\end{align}

Note that $\binom{n}{|B|}/\binom{|A|}{|B|}$ equals $\binom{n}{|A|}/\binom{n-|B|}{n-|A|}$, and thus is non-decreasing as $|B|$ grows from $0$ to $|A|$ (it is, in fact, increasing whenever $|A| \neq n$). If we combine this with $\muf(A)\leq \sum_{j=0}^t \binom{|A|}{j}$, which holds as $\fly$ is \sawt{t}, we get

\begin{align}
|\fly|
&\leq \max_{A\in\fly} \sum_{B \in \fly \cap \pot{A}} \left.\binom{n}{|B|}\right/\binom{|A|}{|B|} \nonumber\\
&\leq \max_{A\in\fly} \sum_{j=\max(0, |A|-t)}^{|A|} \binom{n}{j}\label{eq:layers}\\
&\leq \sum_{j=0}^{t} \binom{n}{\lfloor (n-t)/2\rfloor + j} \label{eq:sizeofA},
\end{align}
where (\ref{eq:sizeofA}) holds because the sum is maximised for $|A| \in \{\lfloor (n+t)/2 \rfloor, \lceil (n+t)/2 \rceil\}$.

The proof is concluded.
\end{proof}

We now turn our attention to characterizing \sawt{t} families that attain maximum size. Similarly to the classification of equality cases in Sperner's theorem \cite{Sperner1928}, the problem is easier for a particular parity of $n$, in our case when $n \equiv t \pmod{2}$.

To simplify the proof, we say that an element $A\in\fly$ is \textit{chain-maximal} if there is a maximal chain $C$ such that $A$ is the maximal element of $C\cap \fly$. In the notation of the proof of \Cref{th:nintersecting} the condition becomes $\mathbb{P}_{\mathcal{C}}(M(\mathcal{C}) = A)>0$. Clearly, all maximal elements of $\fly$ are chain-maximal.

\begin{proof}[Proof of \Cref{th:nintersecting-class}]

First, we note that we may omit the case $t=0$ which describes well-known equality cases in Sperner's theorem \cite{Sperner1928}. In the rest of the proof we assume $t>0$.

The statement is trivial if $t=n$ and straightforward to check for $t=n-1$ since $\pot{[n]} \setminus \{B\}$ is \sawt{t} for any $B \in \pot{[n]}$. This resolves \ref{item:second} and the case $t=n$ in \ref{item:first}. Suppose further that $t\leq n-2$.

Let $\fly$ be a \sawt{t} family of size $\sum_{j=0}^{t} \binom{n}{\lfloor(n-t)/2\rfloor + j}$. Then for $\fly$, (\ref{eq:max}), (\ref{eq:layers}) and (\ref{eq:sizeofA}) are all equalities. Equality in (\ref{eq:max}) means that each of the chain-maximal elements $A \in \fly$ maximises $\mathbb{E}_{\mathcal{C}}[w(|\mathcal{C}\cap \fly|)\, \vert\, M(\mathcal{C}) = A]$. Then, equality in (\ref{eq:layers}) means that for every such element $A\neq [n]$ of $\fly$, the other elements in $\pot{A} \cap \fly$ are exactly the subsets of $A$ in the $t$ layers of the hypercube below $A$. 

Moreover, if $t \equiv n \pmod{2}$, equality in \eqref{eq:sizeofA} is attained if and only if $|A|=(n+t)/2$.
Hence, all maximal elements of $\fly$ have size $(n+t)/2<n$, and all other elements of $\fly$ lie in the $t$ layers of the hypercube beneath the maximal elements.
This proves \ref{item:first}.

Now, assume $t \not \equiv n \pmod{2}$.

By equality in (\ref{eq:sizeofA}), we see that all chain-maximal elements of $\fly$ have size in $\{(n+t-1)/2, (n+t+1)/2\}$.
Therefore, $[n]$ is not an element of $\fly$. Equality in (\ref{eq:layers}) thus implies that for every chain-maximal $A\in\fly$, we have $\pot{A} \cap \fly$ is precisely $\bigcup_{j=|A|-t}^{|A|} (\pot{A} \cap \binom{[n]}{j})$.

If $\fly$ contains no set of size $(n+t+1)/2$, then we immediately get $\fly=\bigcup_{j=0}^{t} \binom{[n]}{(n-t-1)/2 +j}$.

Suppose, finally, that $\fly$ contains a set $A$ of size $(n+t+1)/2$.
Then $A$ is maximal in $\fly$, and thus $\pot{A} \cap \fly=\bigcup_{j=|A|-t}^{|A|} (\pot{A} \cap \binom{[n]}{j})$.  
But then for any $B \subset A$ of size $|A|-1$, it is not true that $\pot{B} \cap \fly=\bigcup_{j=|B|-t}^{|B|} (\pot{B} \cap \binom{[n]}{j})$. Therefore, no such $B$ is chain-maximal. Therefore, all supersets of $B$ of size $(n+t+1)/2$ are in $\fly$. Repeating this argument, one eventually obtains that $\binom{[n]}{(n+t+1)/2)} \subseteq \fly$.
Hence, in this case, $\fly=\bigcup_{j=0}^{t} \binom{[n]}{(n-t+1)/2 +j}$.
This finishes the proof of \ref{item:third} and concludes the proof.
\end{proof}

\section{The maximum size of an intersecting \saw family}

To prove \Cref{th:intersecting-odd} and \Cref{th:intersecting-class} we use a different strategy, analogous to the probabilistic proof of Sperner's theorem using the LYM inequality, see, for example, \cite[pp.237-238]{ProbMethod}. The main ingredient of the proofs will be that a maximal chain on average intersects a \saw family in at most two elements.

Note that in \Cref{le:chain},  there is no condition on intersections. The lemma holds for general \saw families (in particular, also for families that contain no odd-sunflower).

\begin{lemma} \label{le:chain}
Let $\mathcal{C}$ be a maximal chain on $[n]$ chosen uniformly at random. If $\fly$ is a \saw family on $[n]$ which does not contain the empty set, then $\mathbb{E}_{\mathcal{C}}[|\mathcal{C}\cap \fly|]\leq 2$.

Moreover, if equality occurs, then for any chain-maximal $A \in \fly$ we have $\mu_F(A)=|A|+1$ and each set in $\pot{A}\cap \fly$ different from $A$ has size $1$ or $|A|-1$.
\end{lemma}

\begin{proof}
Let $E=\mathbb{E}_{\mathcal{C}}[|\mathcal{C}\cap \fly|]$. Define $M(\mathcal{C})$ to be the maximal set in $\mathcal{C}\cap \fly$ when the intersection is non-empty and $M(\mathcal{C}) = \star$ otherwise. As in the proof of \Cref{th:nintersecting} (again, with the convention that $\mathbb{E}[X\, \vert\, Y]=\mathbb{P}(U\, \vert\, Y)=0$ when $\mathbb{P}(Y)=0$), we can write

\begin{align}
E &= \mathbb{E}_{\mathcal{C}}\Bigl[|\mathcal{C}\cap \fly|\, \Big|\, M(\mathcal{C})= {\star}\Bigr] \cdot\mathbb{P}_{\mathcal{C}}\Bigl(M(\mathcal{C}) = {\star}\Bigr) + \sum_{A\in \fly}\mathbb{E}_{\mathcal{C}}\Bigl[|\mathcal{C}\cap \fly|\, \Big|\, M(\mathcal{C}) = A\Bigr] \cdot\mathbb{P}_{\mathcal{C}}\Bigl(M(\mathcal{C}) = A\Bigr)\nonumber \\
&=\sum_{A\in \fly}\mathbb{E}_{\mathcal{C}}\Bigl[|\mathcal{C}\cap \fly|\, \Big|\, M(\mathcal{C}) = A\Bigr] \cdot\mathbb{P}_{\mathcal{C}}\Bigl(M(\mathcal{C}) = A\Bigr)\nonumber \\
&\leq \max_{A\in\fly}\mathbb{E}_{\mathcal{C}}\Bigl[|\mathcal{C}\cap \fly|\, \Big|\, M(\mathcal{C}) = A\Bigr]\label{al:globalbound} \\
&\leq 1+\max_{A\in\fly} \sum_{\substack{B\in\fly\cap\pot{A} \\ B\neq A }}\mathbb{P}_{\mathcal{C}}\Bigl(B\in\mathcal{C}\, \Big|\, M(\mathcal{C}) = A\Bigr)\nonumber \\
&\leq 1 + \max_{A\in\fly}\sum_{\substack{B\in\fly\cap\pot{A} \\ B\neq A }} \frac{1}{\binom{|A|}{|B|}}. \nonumber
\end{align}

Since $\fly$ does not contain the empty set and $\muf(A)\leq |A|+1$ for any $A$ in $\fly$, we compute

\begin{align}
E &\leq 1 + \max_{A\in\fly}\sum_{\substack{B\in\fly\cap\pot{A} \\ B\neq A }} \frac{1}{\binom{|A|}{|B|}} \nonumber \\
&\leq 1+\max_{A\in\fly}\sum_{\substack{B\in\fly\cap\pot{A} \\ B\neq A }} \frac{1}{|A|} \label{al:binom} \\
&\leq 2, \label{al:localbound} 
\end{align}
as required.

To deduce the `moreover' part, suppose that $\mathbb{E}_{\mathcal{C}}[|\mathcal{C}\cap \fly|]= 2$ and $A\in \fly$ is chain-maximal. In other words, $\mathbb{P}(M(\mathcal{C}) = A)>0$, and thus to get equality in (\ref{al:globalbound}) we need $\mathbb{E}_{\mathcal{C}}[|\mathcal{C}\cap \fly|\, \vert\, M(\mathcal{C}) = A]=2$. We now see that $\mu_F(A)=|A|+1$ and each set in $\pot{A}\cap \fly$ different from $A$ has size $1$ or $|A|-1$ as equality must occur in (\ref{al:localbound}), respectively, (\ref{al:binom}). 
\end{proof}

To find the maximum size of an intersecting \saw family on $[2k-1]$, we combine \Cref{le:chain} and the fact that an intersecting family on $[2k-1]$ contains at most $\binom{2k-1}{k}$ elements from the middle two layers.

\begin{proof}[Proof of \Cref{th:intersecting-odd}]
The claimed maximal cardinality is attained by the intersecting \saw family $\binom{[2k-1]}{k}\cup\binom{[2k-1]}{k+1}$.

To prove the upper bound, let $\fly$ be an intersecting \saw family. Since $\fly$ is intersecting, it does not contain the empty set. By \Cref{le:chain} for a random maximal chain $\mathcal{C}$ we compute

\begin{equation} \label{eq:level bounds}
2\geq \mathbb{E}_{\mathcal{C}}[|\mathcal{C}\cap \fly|] = \sum_{i=0}^{n} \frac{|\fly_i|}{\binom{n}{i}}. 
\end{equation}

As $\fly$ is intersecting, we need $|\fly_{k-1}| + |\fly_{k}|\leq \binom{2k-1}{k}$. The result follows as $\binom{2k-1}{k-1}$, $\binom{2k-1}{k}$ and $\binom{2k-1}{k+1}$ are the three greatest binomial coefficients of the form $\binom{2k-1}{i}$.
\end{proof}

The extremal families are easy to deduce from the proof.

\begin{proof}[Proof of \Cref{th:intersecting-class}]
The statement is clear if $n\leq 3$, so we suppose that $n\geq 5$ onwards. If $\fly$ is an intersecting \saw family of maximal size, then
equality in (\ref{eq:level bounds}) holds. Morevoer, the final paragraph of the proof of \Cref{th:intersecting-odd} implies that $\fly$ consists of one half of the elements of $\binom{[2k-1]}{k+1} \cup \binom{[2k-1]}{k-2}$ and one half of the elements of $\binom{[2k-1]}{k} \cup \binom{[2k-1]}{k-1}$.

As it is intersecting, it contains at least one element of $\binom{[2k-1]}{k+1}$, which for $n>5$ (that is, $k>3$) by the `moreover' part of \Cref{le:chain} contains all its subsets of size $k$. Using the same argument as in the proof of \Cref{th:nintersecting-class} \ref{item:third} we deduce that in fact $\fly$ contains all sets of size $k+1$ and thus it must be $\binom{[2k-1]}{k}\cup\binom{[2k-1]}{k+1}$.

This remains true for $n=5$ unless there is a singleton set, say $\{ 5\}$ in $\fly$. But then all sets of $\fly$ have $5$ as an element and thus $\{ A\setminus \{ 5\} : A\in\fly \}\setminus \{\emptyset\}$ is a \saw family on $[4]$ of size $14$, which contradicts \Cref{th:nintersecting}. 
\end{proof}

For even $n=2k$, the same proof technique one gives an upper bound of $\frac{1}{2}\binom{2k}{k}+\binom{2k}{k+1}+\frac{1}{2}\binom{2k}{k+2}$ for the maximum size of an intersecting \saw family. However, using ad-hoc arguments it can be shown that this bound is not attained for any sufficiently large $n$.

Instead, as stated in \Cref{conj:intersecting-even}, we believe that this maximum size is $\frac{1}{2}\binom{2k}{k}+\binom{2k}{k+1}+\binom{2k-1}{k+2}$, with a lightning being one of the extremal constructions. Note that a lightning would not be a unique extremal construction, as it could be possible to exchange some of the $(k+1)$-sized sets not containing $1$ with their complements.

A possible first step towards \Cref{conj:intersecting-even} is proving it with an extra assumption that the \saw family lies in the union $\binom{[2k]}{k}\cup\binom{[2k]}{k+1}\cup\binom{[2k]}{k+2}$.

\section*{Concluding remarks}
In this paper, we studied \sawt{t} families, with the hope of achieving a better understanding of families which contain no odd-sunflowers. We have given a robust answer to \Cref{qu:nintersecting}. However, the case of \Cref{qu:intersecting} is still mostly open --- of interest would be already the case of $n$ even and $t=1$, see \Cref{conj:intersecting-even}.

We finish on a property of \saw families (and therefore also of families which contain no odd-sunflowers) that could be of further interest, the ``Bubbling-up Lemma''.

\begin{lemma} \label{le:bubbling up}
Let $\fly$ be a \saw family and $C\subset B\subset A$ be sets such that $A$ and $C$ lie in $\fly$ but $B$ does not lie in $\fly$, and $|B|=|A|-1$. Then $\widetilde{\fly} = (\fly \cup \{B\}) \setminus \{C\}$ is a \saw family.
\end{lemma}

\begin{proof}
Let $D$ be a set in $\widetilde{\fly}$. If $D\neq B$, then the number of subsets of $D$ in $\widetilde{\fly}$ is at most the number of subsets of $D$ in $\fly$ since $C\subset B$. If $D=B$, then $\muft(B) \leq \muft(A) -1 \leq |A|+1-1 = |B|+1$. Thus $\widetilde{\fly}$ is a \saw family. 
\end{proof}

\section*{Acknowledgements}
The authors would like to thank Béla Bollobás and Mark Wildon for their valuable comments on the manuscript, including a suggestion to generalise \saw families to \sawt{t} families made by the first-named. The authors would also like to thank Lutz Warnke for pointing out a missing reference in an earlier version.

The first author would like to acknowledge support by the EPSRC (Engineering and Physical Sciences Research Council), reference EP/V52024X/1, and by the Department of Pure Mathematics and Mathematical Statistics of the University of Cambridge. The second author would like to acknowledge support from Royal Holloway, University of London.

\bibliographystyle{abbrvnat}  
\bibliography{bibliography}

\end{document}